\documentclass[10pt,reqno]{amsart}   
\usepackage{amssymb,amscd,latexsym}   
\usepackage{amsmath}
\usepackage{amsthm,times}
\usepackage[all]{xy}
\usepackage{epsfig,graphicx}
\usepackage{graphicx,color}
\usepackage[utf8]{inputenc}
\usepackage{float}
\usepackage{tikz}
\usepackage{enumitem}

\usepackage{tikz-cd}

\newcommand{\llar}{-\kern-5pt-\kern-5pt\longrightarrow}

\usepackage{multirow}
\usepackage{makecell} 

\newtheorem{Theorem}{Theorem} 
\newtheorem{Lemma}[Theorem]{Lemma}

\newtheorem{prop}[Theorem]{Proposition}
\newtheorem{Remark}[Theorem]{Remark}

\newtheorem{defi}[Theorem]{Definition}

\DeclareMathOperator{\Hom}{Hom}
\DeclareMathOperator{\id}{id}

\DeclareMathOperator{\Image}{Im}
\DeclareMathOperator{\coker}{coker}

\DeclareMathOperator{\gl}{GL}

\DeclareMathOperator{\rank}{rank}

\DeclareMathOperator{\End}{End}

\DeclareMathOperator{\E}{\rm E}

\newcommand{\op}[1]{{\mathcal O}_{\mathbb{P}^{#1}}}

\DeclareMathOperator{\rk}{{rk}}
\DeclareMathOperator{\im}{im}

\def\EE{{\mathcal{E}}}

\def\ker{{\rm ker}\,}

\def\rk{{\rm rank}\,}

\def\O{{\cal O}}
\def\restr{{\kern-1pt\restriction\kern-1pt}}

\def\X{{\mathcal X}}
\def\O{{\mathcal O}}
\def\E{{\mathcal E}}

\keywords{}

\subjclass[2010]{}

\begin{document}

\title[Monads and moduli for rank 2 bundles with odd determinant]{Monads and moduli components for stable rank 2 bundles with odd determinant on the projective space }

\author{Aislan Leal Fontes}
\address{Departamento de Matem\'atica, UFS - Campus Itabaiana. Av. Vereador Ol\'impio Grande s/n, 49506-036 Itabaiana/SE, Brazil}
\email{aislan@ufs.br}

\author{Marcos Jardim}
\address{IMECC - UNICAMP \\ Departamento de Matem\'atica \\
Rua S\'ergio Buarque de Holanda, 651\\ 13083-970 Campinas-SP, Brazil}
\email{jardim@unicamp.br}


\maketitle

\begin{abstract}
We propose a three-step program for the classification of stable rank 2 bundles on the projective space $\mathbb{P}^3$ inspired by an article by Hartshorne and Rao. While this classification program has been successfully completed for stable rank 2 bundles with even determinant and $c_2\le5$, much less is known for bundles with odd determinant. After revising the known facts about these objects, we list all possible spectra and minimal monads for stable rank 2 bundles with odd determinant and $c_2\le8$. We provide a full classification of all bundles with positive minimal monads, provide a negative answer to a question raised by Hartshorne and Rao, and describe new irreducible components of the moduli spaces of stable rank 2 bundles with odd determinant and $c_2=6,8$.
\end{abstract}

\tableofcontents

\section{Introduction}

The classification of rank 2 bundles on the projective space $\mathbb{P}^3$ is already a classical problem in algebraic geometry, with first results dating back to the 1960s. Over the past 50 years, three ways of classifying stable rank 2 bundles have emerged:
\begin{enumerate}
\item Barth and Elencwajg introduced in \cite{BE78} the notion of \textit{spectrum} for a rank 2 bundle, see Section \ref{sec:spec} below; thus one can list all possible spectra for rank 2 bundles.
\item Horrocks showed in \cite{Ho64} that every bundle on $\mathbb{P}^3$ is the cohomology of a \textit{monad} consisting only of line bundles, see further details in Section \ref{section1} below; so one can classify rank 2 bundles by classifying all possible monads.
\item the moduli space of stable rank 2 bundles with fixed determinant and second Chern class is potentially composed of several irreducible components, so one can classify rank 2 bundles up to flat deformations.
\end{enumerate}
To complicate matters, there is no straightforward relation between the three classifications: points in a given irreducible component of the moduli space may correspond to bundles with different spectra and monads; a single spectrum is associated to different monads; and the same monad may lead to different spectra. Therefore, a decent description of stable rank 2 bundles on $\mathbb{P}^3$ requires the consideration of all three classification schemes.

To be more precise, let $\mathcal{B}(e,m)$ denote the moduli space of stable rank 2 bundles on $\mathbb{P}^3$ with $c_1=e$ and $c_2=m$. Recall that it is enough to consider $e=0,-1$ (after normalization), and $m>0$ (by Bogomolov's inequality); in addition, $m$ must be even when $e=-1$. Let us revise some of the results available in the literature.

Spectra can be easily determined, and all spectra and possible minimal monads for stable rank 2 bundle with even determinant (the case $e=0$) with $c_2\le8$ were determined by Hartshorne and Rao in \cite[Table 5.3]{HR91}. Moreover, these authors also showed that every possible spectra for rank 2 stable bundles with even determinant is actually realized when $c_2\le19$. Regarding the classification up to flat deformation: 
\begin{itemize}
\item it is more or less clear from the table in \cite[Section 5.3]{HR91} that $\mathcal{B}(0,1)$ and $\mathcal{B}(0,2)$ are irreducible;
\item $\mathcal{B}(0,3)$ and $\mathcal{B}(0,4)$ have two irreducible components each, see \cite{ES} and \cite{C}, respectively; 
\item $\mathcal{B}(0,5)$ has three irreducible components, see \cite{AJTT}.
\end{itemize}

Stable rank 2 bundle with odd determinant (the case $e=-1$) has received considerable less attention:
\begin{itemize}
\item Hartshorne and Sols \cite{HS81} and Manolache \cite{M81} proved that $\mathcal{B}(-1,2)$ is irreducible;
\item Banica and Manolache proved in \cite{BM85} that $\mathcal{B}(-1,4)$ has two irreducible components;
\item Tikhomirov also studied certain families of stable rank 2 bundles with odd determinant, concluding that $\mathcal{B}(-1,6)$ has at least two irreducible components, see \cite[Theorem 3.3]{T14}.
\end{itemize}
As far as we know, no systematic study of the spectrum and monads for stable rank 2 bundles with odd determinant has been attempted so far. The goal of this paper is to fill in this gap and advance on the classification of stable rank 2 bundles of odd determinant with $c_2=6$ and $c_2=8$, while setting up a framework for future studies.

We organize this paper as follows. In Section \ref{section1} we recall the concept of minimal monads and its relation with stable rank 2 vector bundles on $\mathbb{P}^3$ and introduce the concepts of \textit{positive}, \textit{non-negative} and \textit{homotopy free} minimal monad. We take this opportunity to extend a well-known result due to Barth and Hulek \cite{BH78}, and show that every isomorphism of vector bundles on $\mathbb{P}^3$ lifts to an (possibly non-unique) isomorphism of monads (see Proposition \ref{prop5}).

In Section \ref{Serre}, we recall the Serre construction linking rank 2 bundles and space curves, and prove a crucial technical result, Lemma \ref{lema48}, adapting \cite[Lemma 4.8]{HR91} to the case of odd determinant. The following section is devoted to recall the definition of the spectrum for a stable rank 2 bundle with odd determinant, and we list all spectra for bundles with $c_2\leq8$, fulfilling the first (and easiest) step of the classification program.

In Section \ref{sec:monads} we adapt the arguments of \cite[Proposition 3.1]{HR91} to establish Theorem \ref{teo1}, which it is the main tool to find the possible minimal monads whose cohomology bundles have a given spectrum. The results of this section are summarized in Table \ref{table:c2=6}, Table \ref{table:c2=8} and Table \ref{table:5}: the first two tables contain all possible positive minimal monads for bundles with $c_2=6$ and $c_2=8$, while the third table lists all possible non-negative minimal monads for the same second Chern classes. We also show that negative monads are not realized by stable rank 2 bundles with odd determinant and $c_2\le8$, and, after eliminating certain cases, provide a complete classification of the realizable positive minimal monads in the same range, and partially fulfill the second step of the classification program. Finally, we also give a negative answer to a question of Hartshorne and Rao, namely \cite[(Q2), p. 806]{HR91}, showing that a certain spectrum cannot be realized by a stable rank 2 bundle, see Proposition \ref{prop4} for the details.

The third and hardest step of the classification program is considered in Section \ref{se:moduli}. We determine the existence of new irreducible components for $\mathcal{B}(-1,6)$ and $\mathcal{B}(-1,8)$, showing that each of the moduli spaces have at least four irreducible components, see Theorem \ref{m(-1,6)} and Theorem \ref{m(-1,8)} in Section \ref{se:moduli}. The completion of this last step requires a detailed study of non-negative monads and of monads with homotopy, which, rather unfortunately, have eluded the authors' best efforts so far.  

\subsection*{Acknowledgments}
MJ is supported by the CNPQ grant number 302889/2018-3 and the FAPESP Thematic Project 2018/21391-1. The authors also acknowledge the financial support from Coordenação de Aperfeiçoamento de Pessoal de Nível Superior - Brasil (CAPES) - Finance Code 001.


\section{Monads for stable rank 2 bundles with odd determinant} \label{section1}

Recall that a \textit{monad} is a complex of vector bundles of the form
\begin{equation}\label{eq2}
\mathbf{M}:\ \ \  \ \mathcal{C}\stackrel{\beta}{\longrightarrow}\mathcal{B}\stackrel{\alpha}{\longrightarrow}\mathcal{A},
\end{equation}
such that $\beta$ is injective and $\alpha$ is surjective. The sheaf 
$\mathcal{E}:=\ker \alpha/\Image\beta$ is the \textit{cohomology of the monad}. In this paper, we 
always assume that the morphism $\beta$ is locally left invertible, so that $\mathcal{E}$ is a vector 
bundle. In addition, a monad is called \textit{minimal} if no direct summand of $\mathcal{A}$ is the 
image of a line subbundle of $\mathcal{B}$ and if no direct summand of $\mathcal{C}$ goes 
onto a direct summand of $\mathcal{B}$. In addition, $\mathbf{M}$ is said to be
\textit{homotopy free} if
\begin{equation}\label{vanish'} \Hom(\mathcal{B},\mathcal{C})=\Hom(\mathcal{A},\mathcal{B})=0.
\end{equation}

Horrocks proved in \cite{Ho64} that every vector bundle on $\mathbb{P}^3$ (of arbitrary rank) is the cohomology of a minimal monad such that the bundles $\mathcal{C}$, $\mathcal{B}$ and $\mathcal{A}$ are sums of line bundles, see also \cite[Section 3]{BH78} and \cite[Corollary 2.4]{MR2010}. More precisely, every vector bundle $\mathcal{E}$ on $\mathbb{P}^3$ can be obtained as the cohomology of a minimal monad of the form
\begin{equation}\label{monad2}
\bigoplus_{i=1}^{r}\omega_{\mathbb{P}^3}(k_i) \stackrel{\beta}{\longrightarrow} \mathcal{B} \stackrel{\alpha}{\longrightarrow} \bigoplus_{j=1}^{s}\mathcal{O}_{\mathbb{P}^3}(l_j),
\end{equation}
where $\mathcal{B}$ also splits as a sum of line bundles; such a monad is called a \emph{minimal Horrocks monad} for $\mathcal{E}$. The integers $k_i$ are the degrees of a minimal set of generators of the module $H^1_*(\mathcal{E}^\vee\otimes\omega_{\mathbb{P}^3})$, while integers $-l_j$ are the degrees of a minimal set of generators of the module $H^1_*(\mathcal{E})$.

Next, assume that $\mathcal{E}$ and $\mathcal{E}'$ are cohomologies of minimal Horrocks monads $\mathbf{M}$ and $\mathbf{M}'$, respectively, of the form
$$ \mathbf{M} ~~\colon~~ \bigoplus_{i=1}^{r}\omega_{\mathbb{P}^3}(k_i) \stackrel{\beta}{\longrightarrow} \mathcal{B} \stackrel{\alpha}{\longrightarrow} \bigoplus_{j=1}^{s}\mathcal{O}_{\mathbb{P}^3}(l_j) ~~{\rm and}~~ $$
$$ \mathbf{M}' ~~\colon~~ \bigoplus_{i=1}^{r}\omega_{\mathbb{P}^3}(k_i') \stackrel{\beta'}{\longrightarrow} \mathcal{B'} \stackrel{\alpha'}{\longrightarrow} \bigoplus_{j=1}^{s}\mathcal{O}_{\mathbb{P}^3}(l_j') ~~. $$
Then \cite[Theorem 2.5]{MR2010} provides an exact sequence
\begin{equation}\label{sqc-homs}
\bigoplus_{i=1}^{r} H^0(\mathcal{B}^\vee\otimes\omega_{\mathbb{P}^3}(k_i)) \oplus \bigoplus_{j=1}^{s} H^0(\mathcal{B}'(-l_j)) \to
\Hom(\mathbf{M},\mathbf{M}') \to \Hom(\mathcal{E},\mathcal{E}') \to 0
\end{equation}
In particular, every $f\in\Hom(\mathcal{E},\mathcal{E}')$ lifts to a morphism between the corresponding monads; such a lift is unique precisely when the condition in display \eqref{vanish'}, that is:
\begin{equation}\label{vanish2}
H^0(\mathcal{B}^\vee\otimes\omega_{\mathbb{P}^3}(k_i))=H^0(\mathcal{B}'(-l_j))=0 ~~,~~ \forall ~~ i=1,\dots,r ~~{\rm and}~~ j=1,\dots,s.
\end{equation}
We observe that if the vanishing conditions in display \eqref{vanish2} are satisfied, then every \emph{isomorphism} $f:\mathcal{E}\to\mathcal{E}'$ between vector bundles lifts to an \emph{isomorphism} $\boldsymbol{\Psi}:\mathbf{M}\to\mathbf{M}'$ between the corresponding Horrocks monads, cf. \cite[remarks in pages 329 and 332]{BH78}.


Let us now assume that $\mathcal{E}$ is a rank $2$ bundle with $c_1(\mathcal{E})=-1$; this implies that $\mathcal{E}$ admits a \textit{twisted symplectic structure}, that is, there exist an isomorphism $f:\mathcal{E}\longrightarrow\mathcal{E}^\vee(-1)$ such that $f^*(-1)=-f$. Moreover, $f$ induces an isomorphism
$$ H^1(\mathcal{E}(-l)) \simeq H^1(\mathcal{E}^\vee(-l-1)) =   H^1(\mathcal{E}^\vee\otimes\omega_{\mathbb{P}^3}(3-l)) ; $$
this means that if $(a_1,\dots,a_s)$ are the degrees of a minimal set of generators for $H^1_*(\mathcal{E})$, then $(3-a_1,\dots,3-a_s)$ are the degrees of a minimal set of generators for $H^1_*(\mathcal{E}^\vee\otimes\omega_{\mathbb{P}^3})$.

It follows that a minimal Horrocks monad for a rank 2 bundle $\mathcal{E}$ with $c_1(\mathcal{E})=-1$ must be of the form
\begin{equation}\label{m00}
\bigoplus_{i=1}^{s}\op3(-a_i-1)\stackrel{\beta}{\longrightarrow}\mathcal{B} \stackrel{\alpha}{\longrightarrow}\bigoplus_{i=1}^{s}\op3(a_i),
\end{equation}
where $\mathcal{B}$ is a sum of $2s+2$ line bundles with $c_1(\mathcal{B})=s+1$. 
If $M:=H^1_*(\mathcal{E})$ has a minimal free presentation of the form
$$ \cdots\rightarrow F_1\rightarrow F_0\rightarrow M\rightarrow0 , $$
then $\rank(F_0)=s$ and, according to \cite[Proposition 2.2]{R84}, $\rank(F_1)=2s+2$.
Conversely, by repeating the arguments of Hartshorne and Rao in \cite[Proposition 3.2]{HR91}, we get that if $M$ admits a minimal free presentation as above, then $\mathcal{E}$ is given as cohomology of a minimal Horrocks monad as in display \eqref{m00} where $F_0=H^0_*(\mathcal{A})$ and $F_1=H^0_*(\mathcal{B})$; in other words, $\mathcal{A}$ and $\mathcal{B}$ are the sheafifications of the free modules $F_0$ and $F_1$, respectively.

In addition, note that $\mathcal{E}^\vee(-1)$ is the cohomology of the monad dual to the one in display \eqref{m00}, twisted by $\op3(-1)$, that is,
$$ \bigoplus_{i=1}^{s}\op3(-a_i-1)\stackrel{\alpha^*(-1)}{\longrightarrow}\mathcal{B}^\vee(-1) \stackrel{\beta^*(-1)}{\longrightarrow}\bigoplus_{i=1}^{s}\op3(a_i). $$
Since $\mathcal{E}$ and $\mathcal{E}^\vee(-1)$ are isomorphic, we get that $\mathcal{B}$ and $\mathcal{B}^\vee(-1)$ are isomorphic as well. Therefore, if $\op3(k)$ is a summand of $\mathcal{B}$, then $\op3(-k-1)$ must also be a summand of  $\mathcal{B}$, and the latter can be expressed in the following way
$$ \mathcal{B} = \bigoplus_{j=1}^{s+1} \Big( \op3(b_j)\oplus\op3(-b_j-1) \Big) ~~,~~ b_j\ge0. $$

We can summarize our conclusions so far in the following statement.

\begin{Lemma}\label{lema3}
Given a rank 2 bundle $\mathcal{E}$ on $\mathbb{P}^3$ with $c_1(\mathcal{E})=-1$, we can find two sequences of integers $\boldsymbol{a}=(a_1,\dots,a_s)$ and $\boldsymbol{b}=(b_1,\dots,b_{s+1})$ such that $\mathcal{E}$ is the cohomology of a monad of the form
\begin{equation}\label{eq6}
\bigoplus_{i=1}^{s}\op3(-a_i-1) \stackrel{\beta}{\longrightarrow} \bigoplus_{j=1}^{s+1}\Big(\op3(b_j)\oplus\op3(-b_j-1)\Big) \stackrel{\alpha}{\longrightarrow} \bigoplus_{i=1}^{s}\op3(a_i).
\end{equation}
In addition, 
\begin{equation}\label{eq1}
c_2(\mathcal{E})=\displaystyle\sum_{i=1}^{s}a_i(a_i+1)-\displaystyle\sum_{j=1}^{s+1}b_j(b_j+1).
\end{equation}
\end{Lemma}

We order the tuples $\boldsymbol{a}$ and $\boldsymbol{b}$ so that $a_1\le \cdots \le a_s$ and $0\le b_1\le \cdots \le b_{s+1}$. The following definition will become relevant later on.

\begin{defi} \label{+monads} \textup{
A monad as in display \eqref{eq6} is called \textit{positive} if all summands of its right-hand term have positive degree; that is $a_i>0$ for each $i=1,\dots,s$. If all summands of its right-hand term have non-negative degree, then the monad is said to be \textit{non-negative}. If any of summands in the right-hand term is negative, then  the monad is said to be \textit{negative}.
} \end{defi}

Here is a first, immediate application of the non-negativity condition.

\begin{Lemma}\label{no syz deg 0}
The cohomology bundle of a non-negative monad as in display \eqref{eq6} is stable if and only if the morphism $\alpha$ does not admit a syzygy of degree $\ge0$.
\end{Lemma}
\begin{proof}
Recall that a rank 2 bundle $\E$ with $c_1(\E)=-1$ is stable if and only if $H^0(\E(-p))=0$ for every $p\ge0$. From the exact sequence
$$ 0 \to \bigoplus_{i=1}^{s}\op3(-a_i-1-p) \to \ker\alpha(-p) \to \E(-p) \to 0 $$
we can see that if $a_i\ge0$, then $H^0(\E(-p))\simeq H^0(\ker\alpha(-p))$ when $-p\le a_1$; but the latter is just the space of degree $p$ syzygies of $\alpha$, that is, morphisms
$$ \tau:\op3(p)\to\bigoplus_{j=1}^{s+1}\Big(\op3(b_j)\oplus\op3(-b_j-1)\Big)$$
such that $\alpha\circ\tau=0$.
\end{proof}

The next step is to understand the role of the isomorphism $f:\mathcal{E}\to\mathcal{E}^\vee(-1)$ at the level of monads. One can choose a lift of $f$ to a morphism of monads
\begin{equation} \begin{tikzcd}
\bigoplus_{i=1}^{s}\op3(-a_i-1) \arrow{r}{\beta} \arrow[swap]{d}{v} & \mathcal{B}\arrow{r}{\alpha}\arrow[swap]{d}{q}&\bigoplus_{i=1}^{s}\op3(a_i)\arrow[swap]{d}{u} \\%
\bigoplus_{i=1}^{s}\op3(-a_i-1) \arrow{r}{\alpha^*(-1)}& \mathcal{B}^\vee(-1)\arrow{r}{\beta^*(-1)}&\bigoplus_{i=1}^{s}\op3(a_i)
\end{tikzcd} \label{diag-q}
\end{equation}
It follows that the morphism of monads
\begin{equation} \begin{tikzcd}
\bigoplus_{i=1}^{s}\op3(-a_i-1) \arrow{r}{\beta} \arrow[swap]{d}{u^*(-1)} & \mathcal{B}\arrow{r}{\alpha}\arrow[swap]{d}{q^*(-1)}&\bigoplus_{i=1}^{s}\op3(a_i)\arrow[swap]{d}{v^*(-1)} \\%
\bigoplus_{i=1}^{s}\op3(-a_i-1) \arrow{r}{\alpha^*(-1)}& \mathcal{B}^\vee(-1)\arrow{r}{\beta^*(-1)}&\bigoplus_{i=1}^{s}\op3(a_i)
\end{tikzcd} \label{diag-q2}
\end{equation}
is a lift of the isomorphism $f^*(-1):\mathcal{E}\to\mathcal{E}^\vee(-1)$. The equality $f^*(-1)=-f$ implies that $q^*(-1)=-q$ and $v=-u^*(-1)$. However, it is not clear, in general, whether $q$ and $u$ are also isomorphisms.

\begin{Lemma}\label{lem self-dual}
Let $\mathcal{E}$ be a rank 2 bundle with $c_1(\mathcal{E})=-1$ whose Horrocks minimal monad $\mathbf{M}$ is homotopy free. Then $\mathbf{M}$ is isomorphic to a monad of the form
\begin{equation} \label{sd-monad}
\bigoplus_{i=1}^{s}\op3(-a_i-1) \stackrel{\beta}{\longrightarrow} \bigoplus_{j=1}^{s+1}\Big(\op3(b_j)\oplus\op3(-b_j-1)\Big) \stackrel{\beta^*(-1)\circ\Omega}{\longrightarrow} \bigoplus_{i=1}^{s}\op3(a_i),
\end{equation}
so that $\beta$ must satisfy the following equation 
\begin{equation}\label{eq-beta}
\beta^*(-1)\circ\Omega\circ\beta=0 ~~,
\end{equation}
where $\Omega$ is the standard symplectic form.
\end{Lemma}

\begin{proof}
When the Horrocks monad for $\E$ is homotopy free, then, as it was observed above, $q$ and $u$ are isomorphisms. Since  $q^*(-1)=-q$, then we can find $\varphi\in\operatorname{O}(\mathcal{B})$, ie. $\varphi$ is orthogonal ($\varphi^{-1}=\varphi^*$), such that
$$ \varphi \circ q\circ \varphi^* = \Omega = \bigoplus_{j=1}^{s+1} \left( \begin{array}{cc} 0 & 1 \\ -1 & 0 \end{array} \right) , $$
with each $2\times2$ block being regarded as a morphism
$$ \op3(b_j)\oplus\op3(-b_j-1) \longrightarrow \op3(-b_j-1)\oplus\op3(b_j) . $$
In addition, the diagram in display \eqref{diag-q} also implies that $\beta^*(-1)\circ \Omega=u\circ\alpha$, so that $\beta$ must satisfy the following equation in display \eqref{eq-beta}.

Note that the morphism $q$, and therefore also $\varphi$, is uniquely determined by $f$ (which is unique when $\E$ is stable).
\end{proof}


\section{Serre Construction}\label{Serre}

Another way of describing rank 2 bundles on $\mathbb{P}^3$ is provided by the \emph{Serre construction} which, roughly speaking, provides a correspondence between local complete intersection curves and rank 2 vector bundles equipped with a global section. More precisely, let $Y\subset\mathbb{P}^3$ be a local complete intersection curve such that $H^0(\omega_Y(2k+3))$ admits a non-vanishing section $\xi$; since 
$$ H^0(\omega_Y(5-2k)) \simeq {\rm Ext}^1(\mathcal{I}_Y(k-1),\op3(-k)), $$
the section $\xi$ induces and an extension of the form
\begin{equation}\label{eq-HS}
0 \longrightarrow \op3(-k) \stackrel{s}{\longrightarrow} \mathcal{E} \longrightarrow \mathcal{I}_Y(k-1) \longrightarrow 0,
\end{equation}
where $\mathcal{E}$ is a rank 2 bundle with $c_1(\mathcal{E})=-1$ and $s\in H^0(\mathcal{E}(k))$; moreover, 
\begin{equation}\label{c2-degY}
c_2(\mathcal{E})=\deg(Y)+k-k^2.
\end{equation}
One can also check that $\mathcal{E}$ is stable if and only if $Y$ is not contained in a surface of degree $k-1$, and that there is an isomorphism of graded modules
$$ H^1_*(\mathcal{E}) \simeq H^1_*(\mathcal{I}_Y)[k]. $$

Conversely, given a rank 2 bundle $\mathcal{E}$ and a non-trivial section 
$s\in H^0(\mathcal{E}(k))$ that does not vanish in codimension 1, then we obtain an exact sequence like the one in display \eqref{eq-HS}, where $Y=(s)_0$ is the scheme of zeros of $s$. For further details on the Serre construction, see \cite[Theorem 1.1]{H78}.

There are two families of minimal Horrocks monads that have been well explored in the literature. The first one, which we will call \emph{Hartshorne monads} are of the form
\begin{equation} \label{h-monad}
\op3(-2)^{\oplus s} \longrightarrow \op3(-1)^{\oplus s+1} \oplus \op3^{\oplus s+1} \longrightarrow \op3(1)^{\oplus s}.
\end{equation}
If $\mathcal{E}$ is a stable rank 2 bundle given as the cohomology of a monad as above, then $c_1(\mathcal{E})=-1$ and $c_2(\mathcal{E})=2s$. In terms of the Serre construction, the bundles obtained via extensions of the form
$$ 0 \to \op3(-2) \to \mathcal{E} \to \mathcal{I}_Y(1) \to 0 $$
where $Y$ is a disjoint union of $s$ conics, are examples of bundles that arise as cohomology of a monad as in display \eqref{h-monad}.

The second family are the so-called \emph{Ein monads}, see \cite[equation (3.1.A)]{Ein88},
\begin{equation} \label{e-monad}
\op3(-a-1) \longrightarrow \op3(-b_2-1) \oplus \op3(-b_1-1) \oplus \op3(b_1) \oplus \op3(b_2) \longrightarrow \op3(a),
\end{equation}
where $a>b_1+b_2$. If $\mathcal{E}$ is a stable rank 2 bundle given as the cohomology of a monad as display \eqref{e-monad}, then $c_1(\mathcal{E})=-1$ and $c_2(\mathcal{E})=a^2-b_1^2-b_2^2+a-b_1-b_2$. In terms of the Serre construction, such bundles can be represented via extensions of the form
$$ 0 \to \op3(-a-1) \to \mathcal{E} \to \mathcal{I}_Y(a) \to 0 $$
where $Y=Y_1\sqcup Y_2$ with $Y_j$ being a complete intersection curve of bidegree $(a-b_j,a+b_j+1)$, $j=1,2$.

Below is an analogue of \cite[Lemma 4.8]{HR91}, which turns out to be one the most important tools to construct new examples minimal Horrocks monads out of known ones.

\begin{Lemma}\label{lema48}
Let $(\E,\sigma)$ be pair consisting of a stable rank 2 vector bundle $\E$ with $c_1(\E)=-1$ and $c_2(\E)=n$ and a section $\sigma\in H^0(\E(r))$ with $r>0$ such that $X:=(\sigma)_0$ is curve. If $C$ is a complete intersection curve of type $(u,v)$ disjoint from $X$ satisfying $u+v=2r-1$, then there is a pair $(\E',\sigma')$ consisting of a stable rank 2 vector bundle $\E$ with $c_1(\E')=-1$ and $c_2(\E')=n+uv$ and a section $\sigma'\in H^0(\E'(r))$ such that $(\sigma)_0=X\sqcup C$. Moreover, if $\E$ is the cohomology of a minimal monad of the form
$$ \mathbf{M}:\ \ \  \ \mathcal{C} \longrightarrow \mathcal{B} \longrightarrow \mathcal{A}, $$
then $\E'$ is the cohomology of a minimal monad of the form
$$ \mathbf{M}':\ \ \op3(-r)\oplus\mathcal{C} \longrightarrow \op3(r-1-u)\oplus\op3(r-1-v)\oplus\mathcal{B} \longrightarrow \op3(r-1)\oplus\mathcal{A}.$$

\end{Lemma}
\begin{proof}
Let $Y:=X\sqcup C$. Since
$$ \omega_Y=\omega_X\oplus\omega_C = \mathcal{O}_X(2r-5) \oplus \mathcal{O}_C(2r-5) = \mathcal{O}_Y(2r-5), $$
the Serre construction guarantees the existence of a pair $(\E',\sigma')$ with the desired properties. Since $\E$ is stable, the curve $X$ is not contained in a surface of degree $r-1$; thus clearly, $Y$ is not contained in a surface of degree $r-1$ either, so $\E'$ is also stable.

The section $\sigma:\op3(-r)\to \E$ that defines $X$ and the monadic presentation of $\E$ leads to the following short exact sequence
\begin{equation}\label{sqc-ox}
0 \longrightarrow \op3(-r)\oplus\mathcal{C} \longrightarrow
\mathcal{B} \longrightarrow \op3(r-1)\oplus\mathcal{A} \longrightarrow \mathcal{O}_X(r-1) \longrightarrow 0,
\end{equation}
from which one can obtain a minimal presentation for the graded module $H^0_*(\mathcal{O}_X)$.

Next, we consider the short exact sequence
\begin{equation}\label{iy-ic-ox}
0 \to \mathcal{I}_Y(r-1) \to \mathcal{I}_C(r-1) \to \mathcal{O}_X(r-1) \to 0;
\end{equation}
and the Koszul resolution of the complete intersection curve $C$
\begin{equation}\label{res-ic}
0 \to \op3(-r) \to \op3(r-1-u)\oplus\op3(r-1-v) \to \mathcal{I}_C(r-1) \to 0,
\end{equation}
which provides a minimal presentation for the graded module $H^0_*(\mathcal{I}_C)$.

The exact sequence in display \eqref{iy-ic-ox} induces a morphism of graded modules 
$$ H^0_*(\mathcal{I}_C(r-1))\to H^0_*(\mathcal{O}_X(r-1)) $$ whose cokernel is precisely $H^1_*(\mathcal{I}_Y)$. The minimal presentation of the latter can then be deduced from the minimal presentations for $H^0_*(\mathcal{I}_C)$ and $H^0_*(\mathcal{O}_X)$ obtained above; to be precise, we have that
$$ S(r-1-u)\oplus S(r-1-v)\oplus B \longrightarrow S(r-1)\oplus A \longrightarrow H^1_*(\mathcal{I}_Y(r-1)) \to 0 $$
We then just need to recall that $H^1_*(\E')\simeq H^1_*(\mathcal{I}_Y(r-1))$
in order to obtain the desired monadic representation for $\E'$. 
\end{proof}
\bigskip


\section{Spectra for rank 2 stable bundles with odd determinant}\label{sec:spec}

The \textit{spectrum} of a vector bundle $\mathcal{E}$ on $\mathbb{P}^3$ is a sequence of integers that 
encodes partial information about the cohomology modules $H^1_*(\mathcal{E})$ and $H^2_*(\mathcal{E})$. It was originally defined by Barth and Elencwajg \cite{BE78}, and Hartshorne extended this notion to rank 2 reflexive sheaves, see \cite[Theorem 7.1, p.151]{H80}. We will follow Hartshorne's approach, even though we only work with vector bundles.

Let $\mathcal{E}$ be a rank $2$ stable vector bundle on $\mathbb{P}^3$ with $c_1(\mathcal{E})=-1$; set $n:=c_2(\mathcal{E})$, and recall that this is necessarily even. The \textit{spectrum} of $\mathcal{E}$ is an unique multiset of integers $\X(\mathcal{E})=\{k_1,k_2,\cdots,k_n\}$ satisfying the 
following properties
\begin{enumerate}[label=(\alph*)]
\item \label{itemb1} $h^1(\mathbb{P}^3,\mathcal{E}(l)) = h^0(\mathbb{P}^1,\mathcal{H}(l+1))$ for $l\leq-1$ and
\item $h^2(\mathbb{P}^3,\mathcal{E}(l)) = h^1(\mathbb{P}^1,\mathcal{H}(l+1))$ for $l\geq-2$,
\end{enumerate}
where $\mathcal{H}=\bigoplus\op1(k_i)$. 

If a sequence of integers $\{k_1,k_2,\cdots,k_{n}\}$ is the spectrum of a rank $2$ stable bundle 
with $c_1(\mathcal{E})=-1$, then the following properties hold, see \cite[Section 7]{H80}:
\begin{enumerate}[label=\textbf{S.\arabic*}]
\item \label{itema1}$\{k_i\}=\{-k_i-1\}$;
\item \label{itema2}Any integer $k$ between two integers of $\X$ also belongs to spectrum $\X$;
\item\label{itema4} If $k=\max\{-k_i\}$ and there is an integer $u$ with $-k\leq u\leq-2$ that occurs just once in $\X$, then each $k_i\in\X$ with $-k\leq k_i\leq u$ occurs exactly once, see \cite[Proposition 5.1]{H82}.
\end{enumerate}

We list in Table \ref{table:4} all sequences of up to 8 integers satisfying properties \ref{itema1}--\ref{itema4}; we denote $r_j:=\{-j-1,j\}$ and $r_jr_i:=r_j\cup r_i = \{-j-1,j,-i-1,i\}$, and 
so on. These are the possible spectra for stable vector bundles of rank $2$ with odd determinant. However, as we will see in Section \ref{section4}, the conditions \ref{itema1}--\ref{itema4} are not 
sufficient to ensure the existence of a stable rank $2$ vector bundle  with a given spectrum $\X$.

\begin{table}[h]
\begin{tabular}{|c|c|}\hline
$n$ & $\X$\\
\hline
$2$& $\{r_0\}$\\
\hline
4& $\{r_0^2\}, \{r_0r_1\}$\\
\hline
6 & $\{r_0^3\}, \{r_0^2r_1\}, \{r_0r_1^2\}, \{r_0r_1r_2\}$\\
\hline
8& $\{r_0^4\}, \{r_0^3r_1\}, \{r_0^2r_1^2\}, \{r_0^2r_1r_2\}, \{r_0r_1^2r_2\}, \{r_0r_1^3\}, \{r_0r_1r_2r_3\}$\\
\hline
\end{tabular}
\medskip
\caption{Possible spectra for stable rank $2$ vector bundles with $c_1(\mathcal{E})=-1$ and $c_2(\mathcal{E})\leq8$.}
\label{table:4}
\end{table}

\begin{Remark}
We observe that taking $k=3$ and $u=-2$ in property \ref{itema4}, one can eliminate the multiset $\X=\{r_0r_1r_2^2\}$ as a possible spectrum of a stable rank $2$ vector bundle on $\mathbb{P}^3$ with $c_1(\mathcal{E})=-1$ and $c_2(\mathcal{E})=8$.
\end{Remark}

The spectra of a given length can be ordered in the following way: if $\mathcal{X}^n=\{k_1,k_2,\cdots,k_{n}\}$ and ${\mathcal{S}}^n=\{k'_1,k'_2,\cdots,k'_{n}\}$ are spectra with $n$ integers in non-descending order, then
we say $\mathcal{X}^n>{\mathcal{S}}^n$ provided the left-most nonzero integer $k_i-k'_i$ 
is positive. For example, if we look the spectra of the Table \ref{table:4} for $c_2=4$, then we have
$$\mathcal{X}_1^4:=\{r_0^2\}=\{-1,-1, 0,0\}>\{-2, -1, 0, 1\}=\{r_0r_1\}=:\mathcal{X}_2^4.$$
The spectra in the lines of Table \ref{table:4} are listed in decreasing order, from left to right.


\section{Possible minimal monads for a stable bundle with $c_2\leq8$} \label{sec:monads}

Choosing one of the numerical multisets $\X$ displayed in Table \ref{table:4}, our goal in this section is to find all possible minimal Horrocks monads for stable rank $2$ bundles $\mathcal{E}$ on $\mathbb{P}^3$ with fixed spectrum $\X$.

If $s(k_j)$ stands the 
number of repetitions of the integer $k_j$ in the spectrum, then we can write 
\begin{equation}\label{eq3}
\mathcal{X}=\{{(-k-1)}^{s(k)}...,0^{s(0)},\cdots, k^{s(k)}\}.
\end{equation}
Since the first cohomology module $M$ of $\mathcal{E}$ is a graded module over $S=k[X_0,X_1,X_2,X_3]$, we denote by $M_l$ the $l$-th graded piece of $M$; set $m_l:=\dim M_l=h^1(\E(l))$ and let $\rho(l)$ denote the number of minimal generators for $M$ of degree $l$. As $k$ is the largest integer that appears in the spectrum, we have $s(k+1)=0$, and item \ref{itemb1} of the definition of spectrum implies that 
\begin{equation}\label{eq3'}
s(k)=m_{-k-1}=\rho(-k-1).
\end{equation}

This means that the vector bundle $\mathcal{A}$ on the right-hand side of the minimal monad associated to a stable rank $2$ vector bundle $\E$ and spectrum $\X$ as in display \eqref{eq3} has $s(k)$ summands of degree $-k-1$. From definition of spectrum it is easy to see that 
\begin{equation}
    m_{-i}-m_{-i-1}=\displaystyle\sum_{j\geq i-1}s(j), \ i\geq1.
\end{equation}
The tool used for finding the other summands of $\mathcal{A}$ is provided by the result 
below, whose proof is an adaptation of the arguments in \cite[Proposition 3.1]{HR91} for bundles with even determinant.

\begin{Theorem}\label{teo1}
Let $\mathcal{E}$ be a stable vector bundle of rank $2$ on $\mathbb{P}^3$ with $c_1(\mathcal{E})=-1$ and spectrum as in display \eqref{eq3}. 
We have, for $0\leq i < k$,
\begin{equation}\label{ineq1}
s(i)-2\displaystyle\sum_{j\geq i+1}s(j)\leq\rho(-i-1)\leq s(i)-1.
\end{equation}
\end{Theorem}

\begin{proof}
Let $H$ be a general plane in $\mathbb{P}^3$ with equation $x=0$. Starting with the restriction sequence 
$$ 0\longrightarrow\mathcal{E}(-1)\longrightarrow\mathcal{E}\longrightarrow\mathcal{E}_H\longrightarrow0, $$
we consider the rank $2$ vector bundle $\mathcal{E}_H$ on $\mathbb{P}^2\simeq H$ and 
$M=H^1_*(\mathcal{E}_H)$ its first cohomology module. Since $c_1(\mathcal{E})=-1$ it follows 
from \cite[Theorem 3]{Bar77} that the vector bundle $\mathcal{E}_H$ is stable and so the morphism $x:M_{-l-1}\longrightarrow M_{-l}$ is injective for $l\geq0$. If we define $N_{-l}$ of dimension $n_{-l}$ to be the quotient of $M_{-l}$ by the image of $M_{-l-1}$ via the map $x$, then $N$ is a graded submodule of $M$ and $m_{-1}=c_2(\mathcal{E}), m_0=c_2(\mathcal{E})-1$ which implies $0<n_{-2}<c_2(\mathcal{E})$ and $n_{-1}<c_2(\mathcal{E})$. For each $l\geq-1$, we take the natural map
$$ \eta \colon N_{-l-1}\otimes H^0(\op2(1))\longrightarrow N_{-l} $$
and from \cite[Theorem 5.3]{H80} the dimension of the image of $\eta$ is $>n_{-l-1}$ hence, the submodule $N_{-l}$ contains at most $n_{-l}-n_{-l-1}-1$ 
minimal generators of degree $-l$. Since $\rho(-l)$ also is the number of generators of 
$N:=\oplus N_{l}$ (as an $S$-module) in degree $-l$ we get
$$ \rho(-l)\leq n_{-l}-n_{-l-1}-1=s(l-1)-1, \mbox{ for } l\geq1. $$  
On the other hand, the dimension of the image of the map $\eta$ is at most equal to $3n_{-l-1}$, 
hence the submodule $N_{-l}$ must contain at least $n_{-l}-3n_{-l-1}$ minimal generators and we can see that
$$n_{-l}-3n_{-l-1}=s(l-1)-2n_{-l-1}=s(l-1)-2(m_{-l-1}-m_{-l-2})=s(l-1)-2\displaystyle\sum_{j\geq l}s(j),$$
hence the left hand side of the inequality \eqref{ineq1} is true for each $l\geq1$.
\end{proof}

In order to apply our tools, let us start by revisiting the cases $c_2=2$ and $c_2=4$; these have already been explored in the literature and can be found at \cite{BM85,HS81,M81}.

\subsection{Minimal Horrocks monads for $\mathbf{c_2=2}$} \label{subsec:c2=2}

Let $Y$ be the disjoint union of two conics in $\mathbb{P}^3$ and consider the stable rank 2 bundle given by exact sequences of the form
$$ 0\to \op3(-2) \to \mathcal{E} \to \mathcal{I}_Y(1) \to 0 , $$
where $\mathcal{I}_Y$ denotes the ideal sheaf for the curve $Y$. One can check that $c_1(\mathcal{E})=-1$, $c_2(\mathcal{E})=2$, and the only possible spectrum is given by $\X(\mathcal{E})=\{-1,0\}$.

By Theorem \ref{teo1}, we get that $\rho(-1)=1$, and the minimal Horrocks monads must be of the form
\begin{equation*}
\op3(-2) \rightarrow
\op3(-b_1-1)\oplus\op3(b_1)\oplus\op3(-b_2-1)\oplus\op3(b_2)
\rightarrow \op3(1),
\end{equation*} 
where the integers $b_1$, $b_2$ satisfy the relation 
$$ 2 - \big( b_1^2 + b_2^2 + b_1 + b_2 \big) = 2 ; $$
that is, $b_1^2 + b_2^2 + b_1 + b_2=0$. Since $b_1,b_2\ge0$, the only solution of this equation is 
$b_1=b_2=0$, and we end up with a minimal Horrocks monad of the form
\begin{equation}\label{c_2=2}
\op3(-2) \stackrel{\beta}{\longrightarrow} 2\cdot\op3(-1)\oplus2\cdot\op3 \stackrel{\alpha}{\longrightarrow} \op3(1),
\end{equation}
which is a monad of Hartshorne type with parameter $s=1$, compare with \eqref{h-monad}.

Every stable rank 2 bundle $\mathcal{E}$ with $c_1(\mathcal{E})=-1$ and $c_2(\mathcal{E})=2$ is the cohomology of a minimal Horrocks monad as in display \eqref{c_2=2}. Hartshorne and Sols showed in \cite{HS81} that the moduli space $\mathcal{B}(-1,2)$ is an irreducible quasi-projective variety of dimension $11$.


\subsection{Minimal Horrocks monads for $\mathbf{c_2=4}$}\label{subsec:c2=4}

In this case, there are the possible
spectra $\X_1^4=\{{-1}^2, 0^2\}$ and $\X_2^4=\{-2, -1, 0, 1\}$, c.f. the Table \ref{table:4}. 

To show that $\X_1^4=\{{-1}^2, 0^2\}$ is realized as the 
spectrum of a stable $2$ bundle $\mathcal{E}$ on $\mathbb{P}^3$ with $c_1(\mathcal{E})=-1$ and $c_2(\mathcal{E})=4$ it is enough we take a curve $Y$ given by the disjoint union of three conics. Applying Theorem \ref{teo1} (in this case $k=0$ and $\rho(-1)=2$) we conclude that the integer $1$ appears twice in the tuple $\boldsymbol{a}$ and so we get a monad of the form
\begin{equation*}
2\cdot\op3(-2)\rightarrow\bigoplus_{j=1}^{3}\op3(b_j)\oplus\op3(-b_j-1)\rightarrow2\cdot\op3(1),
\end{equation*} 
where the integers $b_1, b_2,b_3$ satisfy the equation 
$b_1^2+b_2^2+b_3^2+b_1+b_2+b_3=0$ which has only $b_1=b_2=b_3=0$ as solution. Therefore, we get the following minimal Horrocks monad
\begin{equation}\label{c2=4}
2\cdot\op3(-2)\stackrel{\beta}{\longrightarrow}3\cdot\op3(-1)\oplus3\cdot\op3\stackrel{\alpha}{\longrightarrow}2\cdot\op3(1),
\end{equation} 
which is a monad of Hartshorne type with parameter $s=2$, compare with \eqref{h-monad}.


The set $\X_2^4=\{-2, -1, 0, 1\}$ can also be realized as the spectrum of a stable bundle. Indeed, it is sufficient to consider a curve $Y=Y_1\sqcup Y_2$ where $Y_1$ is a plane curve of degree $4$ and $Y_2$ is a complete intersection of two surfaces of degree $2$ and $3$. Since $\deg(Y)=10$ and $\omega_Y\simeq\op3(1)$, the rank 2 sheaf given by an extension of the form
\begin{equation}
0\rightarrow\op3(-3)\rightarrow\mathcal{E}\rightarrow\mathcal{I}_Y(2)\rightarrow0.
\end{equation}
is locally free and $c_1(\mathcal{E})=-1$ and $c_2(\mathcal{E})=4$. In particular, $H^1(\mathcal{E}(-2))\neq0$ and from the properties of spectrum we conclude that $\X(\mathcal{E})=\X_2$. Using Theorem \ref{teo1} we find $k=1$, $\rho(-2)=1$, and $\rho(-1)=0$; repeating the arguments above we conclude that $\mathcal{E}$ can be obtained as the cohomology of a minimal Horrocks monad of the form
\begin{equation}\label{cc2=4}
\op3(-3)\stackrel{\beta}{\longrightarrow}\op3(-2)\oplus\op3(-1)\oplus\op3\oplus\op3(1)\stackrel{\beta}{\longrightarrow}\op3(2),
\end{equation} 
which is a monad of Ein type with parameters $a=1$, $b_1=0$ and $b_2=1$, compare with \eqref{e-monad}.

We conclude that every stable rank 2 bundle in $\mathcal{B}(-1,4)$ is the cohomology of a minimal Horrocks monad either of the form \eqref{c2=4} or \eqref{cc2=4}. In addition, Banica and Manolache showed in \cite{BM85} that the moduli space $\mathcal{B}(-1,4)$ has two connected components: $\mathcal{B}_1(-1,4)$ parametrizes those bundles with spectrum $\X_1$ and has dimension $27$, while $\mathcal{B}_2(-1,4)$ parametrizes the bundles with spectrum $\X_2$ and has dimension $28$.


\subsection{Possible positive minimal monads for $\mathbf{c_2=6}$}  \label{subsec:c2=6} 

We will now use Theorem \ref{teo1} to determine all possible \emph{positive} minimal Horrocks monads for stable rank 2 bundles $\mathcal{E}$ with $c_1(\mathcal{E})=-1$ and $c_2(\mathcal{E})=6$. 

By repeating the arguments in Sections \ref{subsec:c2=2} and  \ref{subsec:c2=4} we get that the spectra $\mathcal{X}_1^6=\{r_0^3\}$ and 
$\mathcal{X}_4^6=\{r_0r_1r_2\}$ provide, respectively, the following 
possible positive minimal Horrocks monads 
\begin{equation*}
3\cdot\op3(-2)\rightarrow4\cdot\op3(-1)\oplus4\cdot\op3\rightarrow3\cdot\op3(1)
\end{equation*}
and
\begin{equation*}
\op3(-4)\rightarrow\op3(-3)\oplus\op3(-1)\oplus\op3\oplus\op3(2)\rightarrow\op3(3);
\end{equation*}
these are monads of Hartshorne (with parameter $s=4$) and Ein type (with parameters $a=3$, $b_1=0$ and $b_2=2$), respectively.

By applying Theorem \ref{teo1} to the spectrum $\X_2^6=\{r_0^2r_1\}$, it follows that $k=1$, $\rho(-2)=1$ 
and $0\leq\rho(-1)\leq1$.  If $\rho(-1)=0$, then the possible positive minimal Horrocks monad has the form
\begin{equation*}
\op3(-3)\rightarrow\op3(b_1)\oplus\op3(-b_1-1)\oplus\op3(b_2)\oplus\op3(-b_2-1)\rightarrow\op3(2),
\end{equation*}
such that $b_1^2+b_2^2+b_1+b_2=0$, so $b_1=b_2=0$; notice that the latter is an Ein type monad.

When $\rho(-1)=1$, by repeating the above arguments we get the possible positive monad 
\begin{equation*}\label{monad:p3}
\op3(-3)\oplus\op3(-2)\rightarrow\op3(-2)\oplus2\cdot\op3(-1)\oplus2\cdot\op3\oplus\op3(1)\rightarrow\op3(1)\oplus\op3(2).
\end{equation*}

Finally, for the spectrum $\X_3^6=\{r_0r_1^2\}$ we obtain $k=1$, $\rho(-2)=2$ and $\rho(-1)=0$, hence
the monad corresponding has the form
\begin{equation}\label{monad:p5}
2\cdot\op3(-3)\rightarrow\bigoplus_{j=1}^{3}\op3(b_j)\oplus\op3(-b_j-1)\rightarrow2\cdot\op3(2),
\end{equation} 
where the integers $b_1, b_2, b_3$ satisfy the equation $b_1^2+b_2^2+b_3^2+b_1+b_2+b_3=6$ which 
has how solutions $b_1=b_2=b_3=1$ or $b_1=b_2=0$ and $b_3=2$. 

By fixing $\boldsymbol{a}=(a_1\dots a_s):=a_1^{r_1},\cdots, a_k^{r_k}$ and $\boldsymbol{b}=(b_1,\dots,b_{s+1}):=b_1^{t_1},\cdots, b_w^{t_w}$ where $a_i^{r_i}$ and $b_j^{t_j}$ stand for the repetition of $a_i$ and $b_j$ in $\boldsymbol{a}$ and $\boldsymbol{b}$ respectively, the results of this section are summarized in Table \ref{table:c2=6}. We will see below that monads of type (P3) and (P4) with stable cohomology bundles do exist, see Theorem \ref{thm:p3+p4}. On the other hand, there are no stable rank 2 bundles that are the cohomology of a monad of type (P5), compare with Lemma \ref{lema1}. 

\begin{table}[ht]
\begin{center}
\begin{tabular}{ | c | c | c |c| c|} 
\hline
Spectrum & $\boldsymbol{b}$ & $\boldsymbol{a}$ & Label & Homotopy free? \\ 
\hline
$\X_1^6=\{r_0^3\}$ & $0^4$ & $1^3$ & (P1) & yes\\
\hline
$\X_2^6=\{r_0^2r_1\}$ & \makecell{$0^2$ \\ $0^2, 1$} & \makecell{$2$ \\ $1,2$} & \makecell{(P2) \\ (P3)} & \makecell{yes \\ no}\\
\hline
$\X_3^6=\{r_0r_1^2\}$ & \makecell{$1^3$\\\textcolor{red}{$0^2,2$}} & \makecell{$2^2$\\\textcolor{red}{$2^2$}} & \makecell{(P4) \\ \textcolor{red}{(P5)}} & \makecell{yes \\ no}\\
\hline
$\X_4^6=\{r_0r_1r_2\}$ & $0,2$ & $3$ & (P6) & yes\\
\hline
\end{tabular}
\bigskip
\caption{Spectra and possible positive minimal monads for stable rank 2 bundles with $c_1=-1$ and $c_2=6$. The 
columns $\boldsymbol{b}$ and $\boldsymbol{a}$ indicate, respectively, the non-negative degrees in the 
middle term and the degrees in the right-hand term of a minimal monad whose cohomology bundle has the 
given spectrum. Monads of type (P5) do not admit stable cohomology bundles.}
\label{table:c2=6}
\end{center}
\end{table}


\subsection{Possible positive minimal monads for $\mathbf{c_2=8}$}


Following the cases in Table \ref{table:4}, let us now list the possible positive minimal monads when 
$c_2=8$. The first examples are the following  Hartshorne (with parameter $s=5$) and 
Ein (with parameters $a=4$, $b_1=0$ and $b_2=3$) type monads:
\begin{equation*}
4\cdot\op3(-2)\rightarrow5\cdot\op3(-1)\oplus5\cdot\op3\rightarrow4\cdot\op3(1).
\end{equation*}
and
\begin{equation*}
\op3(-5)\rightarrow\op3(-4)\oplus\op3(-1)\oplus\op3\oplus\op3(3)\rightarrow\op3(4).
\end{equation*}
One can check that the cohomology bundles of such positive monads have spectra equal to $\X_1^8=\{r_0^4\}$ and $\X_7^8=\{r_0r_1r_2r_3\}$, respectively. 

For the spectrum $\X^8_2=\{r_0^3r_1\}$ we have $k=1$, and it follows from Theorem \ref{teo1} that $\rho(-2)=1$ and $\rho(-1)\in\{1, 2\}$.
When $\rho(-1)=1$ we must find non-negative integers satisfying the equation $$ b_1^2+b_2^2+b_3^2+b_1+b_2+b_3=0; $$ 
the only solution is $b_1=b_2=b_3=0$. When $\rho(-1)=2$, equation \eqref{eq1} becomes 
$$ b_1^2+b_2^2+b_3^2+b_4^2+b_1+b_2+b_3+b_4=2 $$
whose only non-negative solution is $b_1=b_2=b_3=0$ and $b_4=1$. 

For the next spectrum, namely $\X^8_3=\{r_0^2r_1^2\}$, we observe that $k=1$, hence $\rho(-2)=2$ 
and either $\rho(-1)=0$ or $\rho(-1)=1$. By analyzing the first possibility we get the equation
$$ b_1^2+b_2^2+b_3^2+b_1+b_2+b_3=4 $$ 
whose only non-negative integral solution is $b_1=0$ and $b_2=b_3=1$. In the other case, the equation in display \eqref{eq1} becomes
$$ b_1^2+b_2^2+b_3^2+b_4^2+b_1+b_2+b_3+b_4=6, $$
and it admits two non-negative integral solutions: $(b_1,b_2,b_3,b_4)=(0,1,1,1)$ and \linebreak $(b_1,b_2,b_3,b_4)=(0,0,0,2)$

The spectrum $\X^8_4=\{r_0^2r_1r_2\}$ is such that
$k=2$, $\rho(-3)=1$, $\rho(-2)=0$, and either $\rho(-1)=0$ or $\rho(-1)=1$. If $\rho(-1)=0$, then the equation to be solved is 
$$ b_1^2+b_2^2+b_1+b_2=4; $$
its only non-negative integral solution is $b_1=b_2=1$. Regarding the case $\rho(-1)=1$, the equation in display \eqref{eq1} becomes $$ b_1^2+b_2^2+b_3^2+b_1+b_2+b_3=6; $$
the latter admits two non-negative integral solutions: $b_1=b_2=b_3=1$ and $b_1=b_2=0$ and $b_3=2$.

Next, considering the spectrum  $\X^8_5=\{r_0r_1^2r_2\}$, then $k=2$, and $\rho(-3)=1$; Theorem \ref{teo1} implies that $\rho(-1)=0$ and either $\rho(-2)=0$ or $\rho(-2)=1$. When $\rho(-2)=0$, one can check that the integers $b_1=b_2=1$ satisfy the equation in display \eqref{eq1}; when $\rho(-2)=1$, then \eqref{eq1} has the 
form
$$ b_1^2+b_2^2+b_3^2+b_1+b_2+b_3=10 $$
and its only non-negative integral solution is $b_1=b_2=1, b_3=2$.

Finally, for the last spectrum $\X^8_6=\{r_0r_1^3\}$ we find $k=1$ with $\rho(-2)=3$ and $\rho(-1)=0$, hence the equation in display
\eqref{eq1} has only one non-negative integral solution, given by $b_1=0$, $b_2=b_3=1$ and $b_4=2$.

We collect all of the information obtained above in Table \ref{table:c2=8}. In Section \ref{section4}, we will show that 
the monads marked in red do not admit stable rank 2 bundles as cohomology.



%
 

	\begin{table}[ht]
		\begin{center}
	\begin{tabular}{ | l | c | c | c |c| } 
\hline
Spectrum & $\boldsymbol{b}$ &  $\boldsymbol{a}$ & Label & Homotopy free?\\ 
\hline
$\X^8_1=\{r_0^4\}$&$0^5$ &$1^4$ &  (P7)& yes\\
\hline
$\X^8_2=\{r_0^3r_1\}$& \makecell{$0^3$\\ $0^3, 1$}& \makecell{$1, 2$\\ $1^2,2$} & \makecell{(P8)\\ (P9)}& \makecell{yes\\ no}\\
\hline
$\X^8_3=\{r_0^2r_1^2\}$& \makecell{$0, 1^2$\\ \textcolor{red}{$0^3, 2$}\\ $0, 1^3$}&  \makecell{$2^2$\\ \textcolor{red}{$1, 2^2$}\\ $1, 2^2$}& \makecell{(P10)\\ \textcolor{red}{(P11)}\\ (P12)}& \makecell{yes\\no\\no }\\
\hline
$\X^8_4=\{r_0^2r_1r_2\}$ &\makecell{$1^2$\\$0^2, 2$\\ \textcolor{red}{$1^3$}} & \makecell{$3$\\ $1, 3$\\ \textcolor{red}{$1, 3$}}& \makecell{(P13)\\ (P14)\\ \textcolor{red}{(P15)}}& \makecell{yes\\no \\no}\\
\hline
$\X^8_5=\{r_0r_1^2r_2\}$&\makecell{$1^2$\\ $1^2, 2$}& \makecell{$3$\\ $2, 3$} & \makecell{(P16)\\ (P17)}&\makecell{yes\\ no}\\
\hline
 $\X^8_6=\{r_0r_1^3\}$&\textcolor{red}{$0, 1^2, 2$} &\textcolor{red}{$2^3$} & \textcolor{red}{(P18)}&no \\
\hline
$\X^8_7=\{r_0r_1r_2r_3\}$ &$0, 3$ &$4$ &(P19)& yes\\
\hline
\end{tabular}
\bigskip
\caption{Spectra and possible positive minimal monads for stable rank 2 bundles with $c_1=-1$ and $c_2=8$. The columns $\boldsymbol{b}$ and $\boldsymbol{a}$ indicate, respectively, the non-negative degrees in the middle term and the degrees in the right-hand term of a minimal monad whose cohomology bundle has the given spectrum. Monads of type (P11) and (P18) do not admit stable cohomology bundles while monads of type (P15) do not exist.}
\label{table:c2=8}
\end{center}
\end{table}
\newpage


\section{Non-negative minimal monads for $c_2\le8$}

\subsection{Non-existence of negative minimal monads}

We will prove in this subsection that there is no stable rank $2$ bundle $\mathcal{E}$ with $c_1(\mathcal{E})=-1$ and $c_2(\mathcal{E})\leq8$ which is the 
cohomology of a negative minimal Horrocks monad of the form \eqref{eq6}. In order to attain this goal we need, besides Theorem \ref{teo1}, the following proposition. 

\begin{prop}\label{propa2}
Consider a minimal Horrocks monad as in display \eqref{eq6}. 
If $\mathcal{A}=\oplus_{i=1}^{s}\op3(a_i)$ has $r$ summands with degrees $\leq l$, then $\mathcal{B}=\bigoplus_{j=1}^{s+1}\Big(\op3(b_j)\oplus\op3(-b_j-1)\Big)$ must contains at least $r+3$ summands with degrees $\geq-l$.
\end{prop}

\begin{proof}
The vector bundle $\mathcal{A}^\vee(-1)$ contains $r$ summands with degrees $\geq -l-1$ and by the 
minimality of the monad these ones can be embedded in a summand of $\mathcal{B}$ whose terms have 
degrees $\geq-l$. The quotient of this embedding has rank at least $3$.   
\end{proof}
It is important to observe that for any $c_2$ the Proposition \ref{propa2} implies that there is not negative 
monad whose terms on the right-hand side are all of non-positive degree. This means that a negative minimal 
monad is obtained from a positive minimal monad by adding terms of negative degree.
\begin{Theorem}\label{negative}
Let $\EE$ be a stable $2$ vector bundle on $\mathbb{P}^3$ with $c_1=-1$ and $c_2\leq8$. There are no 
negative minimal Horrocks monads having $\EE$ as its cohomology bundle.
\end{Theorem}

\begin{proof}
If we apply the Proposition \ref{propa2} with $l=-1$ and $r=1$, then it follows from Table \ref{table:c2=6} and Table \ref{table:c2=8} (obtained of the Theorem \ref{teo1}) that a possible negative minimal monad can be obtained only when we add a summand $\op3(t), t\leq-1$, to the right-hand term $\mathcal{A}$ of one of the monads: (P4), (P12), (P15), (P17) or (P18). By analyzing the equation \eqref{eq1} when we add the summand $\mathcal{O}(t)$ to the bundle $\mathcal{A}$ of one of the first four possible monads above, automatically we add a vector bundle $\mathcal{O}(b)\oplus\mathcal{O}(-b-1)$ to the middle term and we get the equality 
$$b(b+1)=t(t+1),$$
which has as solution $b=-t-1$. According to the notation in \eqref{eq6} and from minimality 
of the monad we observe that the map $\alpha$ has a row with just three entries nonzero, more precisely, if we 
add $\op3(t)$ to (P4), (P15) or (P17) (and also for the monad (12) but its matrix has order $4\times10$) then we obtain 
$$\alpha=\left(\begin{array}{cccccccc}
0&0&0&0&0&p&q&s\\
\ast&\ast&\ast&\ast&\ast&\ast&\ast&\ast\\
\ast&\ast&\ast&\ast&\ast&\ast&\ast&\ast\\
\end{array}\right),$$
where $p,q, s$ are homogeneous polynomials of $k[x,y,z,w]$. If we take a point $p\in \mathcal{V}(p)\cap \mathcal{V}(q)\cap \mathcal{V}(s)$, then the 
$\rank \alpha(p)<3$ and so the morphism $\alpha$ is not surjective which is a contradiction. If we add a summand $\mathcal{O}(t)$ to the righ-hand term of the possible monad (18), then the matrix of $\alpha$ has a zero column and so it is not surjective.
\end{proof}

On the other hand, we can to show that there is always stable $2$ vector bundle on $\mathbb{P}^3$ with $c_2\geq10$ 
which is given as cohomology of a negative minimal Horrocks monad. For this, let's consider $X$ be a locally complete 
intersection curve in $\mathbb{P}^3$ of degree $d$. By assuming that $X$ is union of irreducible 
non-singular curves meeting quasi-transversely in any number of point, from the 
\cite[Proposition 2.8]{HR91}, we can consider a nowhere vanishing section of the sheaf 
$\mathcal{N}_X\otimes\omega_X(1)$ and let $Y$ be the multiplicity $2$ scheme associated to the exact sequence
\begin{equation}\label{exact1}
0\longrightarrow\mathcal{I}_Y\longrightarrow\mathcal{I}_X\longrightarrow\omega_X(1)\longrightarrow0.
\end{equation}
By Ferrand's Theorem c.f. \cite[Theorem 1.5]{H78} we get $\omega_Y=\mathcal{O}_Y(-1)$ and so follows that the 
curve $Y$ is the scheme of zeros of a section $s$ of a vector bundle $\mathcal{E}(2)$ of rank 
$2$ on $\mathbb{P}^3$ where $s$ induces the exact sequence
\begin{equation}\label{exact2}
0\longrightarrow\op3(-2)\longrightarrow\mathcal{E}\longrightarrow\mathcal{I}_Y(1)\longrightarrow0.
\end{equation}
We observe that the vector bundle $\mathcal{E}$ obtained in this manner is such that 
$c_1(\mathcal{E})=-1$ and $c_2(\mathcal{E})=2d-2$ and furthermore it is stable, since $Y$ 
has degree even and so $Y$ cannot be in a plane.

\begin{prop}\label{degree0}
Let $X$ be a curve of type $(a,a+t)$ on a non-singular quadric $Q$ with $a, t\geq2$. The
stable vector bundle $\mathcal{E}$ obtained of $X$ as above has cohomology module $H^1_*(\mathcal{E})$ with 
$t-1$ generators in degree $a-1$ this is $\EE$ is given as cohomology of a 
negative minimal Horrocks monad. 
\end{prop}
\begin{proof}
From the exact sequence
\begin{equation*}
0\longrightarrow\mathcal{O}_Q(-a, -a-t)\longrightarrow\mathcal{O}_Q\longrightarrow\mathcal{O}_X\longrightarrow0,
\end{equation*}
with $\mathcal{I}_X=\mathcal{O}_Q(-a, -a-t)$ we see that the smallest integer $n$ such that $h^1(\mathcal{I}_X(n))\neq0$ is $n=a$ and in this case c.f \cite[Exercice 5.6, 231]{H77}
$$h^1(\mathcal{I}_X(a))=h^1(Q,\mathcal{O}_Q(0,-t))=t-1.$$
Now, from exact sequence \eqref{exact2} we get $H^1(\mathcal{E}(a-1))\simeq H^1(\mathcal{I}_Y(a))$ and 
replacing this one on the exact sequence \eqref{exact1} we obtain
\begin{equation*}
 H^0(\omega_X(a+1))\longrightarrow H^1(\mathcal{E}(a-1))\longrightarrow H^1(\mathcal{I}_X(a))\longrightarrow0,
\end{equation*}
which implies $h^1(\mathcal{E}(a-1))=t-1$.
\end{proof}


\subsection{Possible non-negative minimal Horrocks monads}

We will list all possible non-negative minimal monads whose cohomology is a stable vector bundle on $\mathbb{P}^3$ with 
$c_1=-1$ and $c_2=6,8$. Combining the Theorem \ref{teo1} (or the Tables \ref{table:c2=6} and \ref{table:c2=8}) with the Proposition \ref{propa2} we get the following possible non-negative minimal monads
\begin{table}[ht]
	\begin{center}
		\begin{tabular}{| c | c | c | c|} 
			\hline
			Spectrum& $\boldsymbol{b}$& $\boldsymbol{a}$& Label\\ 
			\hline\hline
		     $\mathcal{X}_2^6$ & $0^3, 1$& $0,1,2$&(N1)\\
			\hline
			$\mathcal{X}_3^6$ &\makecell{$0, 1^3$\\{\color{red}$0^3, 2$}}&\makecell{$0, 2^2$\\{\color{red}$0, 2^2$}}& \makecell{(N2)\\{\color{red}(N3)}}\\
			\hline\hline
		    $\mathcal{X}_2^8$&  \makecell{$0^4$\\ $0^4, 1$}&\makecell{$0, 1, 2$\\ $0, 1^2,2$}& \makecell{(N4)\\(N5)}\\
			\hline
			 $\mathcal{X}_3^8$ &\makecell{$0^2, 1^2$\\{\color{red} $0^4, 2$}\\ $0^2, 1^3$}&\makecell{$0, 2^2$\\ {\color{red}$0, 1, 2^2$}\\ $0, 1, 2^2$}& \makecell{(N6)\\{\color{red}(N7)}\\(P8)}\\
			\hline
			$\mathcal{X}_4^8$ & \makecell{$0^3, 2$\\ $0, 1^3$}&\makecell{$0,1, 3$\\ $0, 1, 3$}& \makecell{(N9)\\(N10)} \\
			\hline
			$\mathcal{X}_5^8$&$0, 1^2, 2$& $0, 2, 3$& (N11)\\
			\hline
			 $\mathcal{X}_6^8$&\makecell{{\color{red}$0^2, 1^2, 2$}\\ {\color{red}$1^5$}} & \makecell{{\color{red}$0, 2^3$}\\{\color{red}$0, 2^3$}}&\makecell{{\color{red}(N12)}\\{\color{red}(N13)}}\\
			\hline
		\end{tabular}
		\bigskip
		\caption{Possible non-negative minimal monads for stable rank 2 bundles with $c_1=-1$ and $c_2=8$. The monads marked in red do not admit stable cohomology bundles (if they exist). None of the monads listed in the table is homotopy free.}
		\label{table:5}
	\end{center}
\end{table}
\begin{Remark}\label{obs1}
If a minimal monad of the type labeled (N3), (N7) or (N12) in Table \ref{table:5} exists, then the matrix of the morphism $\alpha$ which form the monad contains a column of zeros and therefore this morphism is not surjective, leading to a contradiction. Therefore the possibilities (N3), (N7) and (N12) do not provide non-negative minimal monads.

For example, the morphism $\alpha$ in (N3) is
$$\alpha:\op3(2)\oplus3\cdot\op3\oplus3\cdot\op3(-1)\oplus\op3(-3)\longrightarrow2\cdot\op3(2)\oplus\op3$$
and the minimality of the monad implies the first column of $\alpha$ to be zero.
\end{Remark}

\begin{prop}\label{propa3}
If the bundle $\E$ is the cohomology of the non-negative minimal monad (N13) of the form
$$ \op3(-1)\oplus3\cdot\op3(-3) \stackrel{\alpha}{\longrightarrow}
5\cdot\op3(-2)\oplus5\cdot\op3(1) \stackrel{\beta}{\longrightarrow}
\op3\oplus3\cdot\op3(2),
$$
then $\E$ cannot be stable.
\end{prop}

\begin{proof}
Since the monad is minimal, the image under $\beta$ of the summand $5\cdot\op3(1)$ in the middle term must be contained in $3\cdot\op3(2)$. Letting $K:=\ker\beta$ we get a commutative diagram
$$ \xymatrix{
	& 0 \ar[d] & 0 \ar[d] & 0 \ar[d] & \\
	0 \ar[r]& K' \ar[r] \ar[d]^{\tau} & 5\cdot\op3(1) \ar[r]^{\beta'} \ar[d] & 3\cdot\op3(2) \ar[d] & \\
	0 \ar[r] & K \ar[r]\ar[d] & 5\cdot\op3(-2)\oplus5\cdot\op3(1) \ar[r]^{\beta} \ar[d] & \op3\oplus3\cdot\op3(2) \ar[d]\ar[r] & 0  \\
	0 \ar[r] & K'' \ar[r] & 5\cdot\op3(-2) \ar[r]^{\beta''}\ar[d] & \op3 \ar[d]\ar[r] & 0  \\
	& & 0 & 0 & 
} $$
where $\beta'$ is the restricted morphism.
Note that
$$ h^0(\E)=h^0(K)=h^0(K') ~~{\rm and}~~ h^0(\E(1))=h^0(K(1))-1=h^0(K'(1))-1. $$

We assume that $\E$ is stable, so that $h^0(\E)=0$, and we show that this leads to a contradiction.

Set $I:=\im(\beta')$, and note that this is a subsheaf of $3\cdot\op3(2)$, thus, by \cite[Lemma 3.7]{HR91}, $h^0(I(1))\le 20$ if $\rk(I)=1$ and $h^0(I(1))\le 40$ if $\rk(I)=2$. Moreover, we have that $h^0(K'(1)) \ge 50 - h^0(I(1))$, thus $h^0(\E(1))\ge9$ when $\rk(I)\le2$. Taking a non trivial $s\in H^0(\E(1))$, let $X=(s)_0$ be its zero locus; this is a curve of degree 8. From the exact sequence
$$ 0 \to \op3(-1) \to \E \to I_X \to 0 $$
we get that $h^0(I_X(1))=h^0(\E(1))-1\ge8$, meaning that $X$ is contained in at least 8 distinct planes, which is clearly impossible. It follows that $\rk(I)=3$; since $\mu(I)\le2$, we obtain $c_1(I)\le6$.

Since $h^0(K''(1))=0$, the composition
$$ \op3(-1) \to \op3(-1)\oplus3\cdot\op3(-3) \stackrel{\alpha}{\rightarrow} K 
\rightarrow K''$$
must vanish, thus the morphism $\op3(-1) \to \op3(-1)\oplus3\cdot\op3(-3) \stackrel{\alpha}{\rightarrow} K$ must factor through $K'$, leading us to the following commutative diagram
$$ \xymatrix{
	& 0 \ar[d] & 0 \ar[d] &  & \\
	0 \ar[r]& \op3(-1) \ar[r]^{\sigma} \ar[d] & K' \ar[r] \ar[d]^{\tau} & I_Y(l) \ar[r]\ar[d] & 0\\
	0 \ar[r] & \op3(-1)\oplus3\cdot\op3(-3) \ar[r]^{~~~~~~~~~~\beta}\ar[d] & K \ar[r]\ar[d] & E \ar[r] & 0  \\
	 & 3\cdot\op3(-3) \ar[r]\ar[d] & \coker\tau \ar[d] &  &   \\
	& 0  & 0 &  &
} $$
where $\sigma$ is the induced non trivial section in $H^0(K(1))$, $Y$ is its zero locus (remark that $Y$ is a curve because $h^0(K)=0$ by hypothesis), and $l:=c_1(K)+1$, so that $\coker\sigma$ is isomorphic to the twisted ideal sheaf $I_Y(l)$. 


Let $B$ be the kernel of the morphism in the third line of the previous diagram; since $\coker\tau$ is torsion 
free (it is a subsheaf of $K''$, which is locally free), $B$ must be reflexive. We get, from the snake lemma, the exact sequence
$$ 0 \to B \to I_Y(l) \to \E. $$
If $B\neq0$, then $\rk(B)=1$ and we conclude that in fact $B$ is a line bundle, and it follows that the 
quotient $I_Y(l)/B$ is a torsion sheaf; but $I_Y(l)/B$ is a subsheaf of $\E$, so we must conclude that $B=0$, and $I_Y(l)$ is a subsheaf of $\E$.

The stability of $\E$ implies that $l\le-1$, thus in fact $c_1(K)\le-2$. It then follows that $c_1(I)=5-c_1(K)\ge7$, thus providing the desired contradiction. 
\end{proof}

It is not clear to us whether or not the other non-negative monads listed in Table \ref{table:5} (that is, the ones not marked in red) actually exist.


\section{ Classification of the positive minimal monads for $c_2\leq8$}\label{section4}

While we listed in Section \ref{sec:monads} all possible (positive and non-negative) minimal Horrocks monads, we will now present that all positive minimal monads for stable rank 2 bundles with $c_1=-1, c_2\leq8$ that actually exist. We also answer to a question left by Hartshorne \cite[(Q2), p. 806]{HR91} about the conditions \ref{itema1}--\ref{itema4} on an integer sequence $\X$ being sufficient for the existence of a stable vector bundle with spectrum $\X$, see Proposition \ref{prop4}. 

Following the labeling presented in Tables \ref{table:c2=6} and \ref{table:c2=8}, we first observe that (P1) and (P7) are Hartshorne monads and (P2), (P6), (P13), (P16), (P19) are Ein monads, see Section \ref{Serre} above; we will therefore focus on the remaining possibilities. We start with a non-existence result.

\begin{Lemma}\label{lema1}
There are no minimal monads of type (P5), (P11) and (P18), as labeled in Tables \ref{table:c2=6} and \ref{table:c2=8}, whose cohomology bundle is stable.
\end{Lemma}
\begin{proof}
We work with the monad of type (P5); the other two cases can be argued in a similar way.

Suppose, by contradiction, that there exists a minimal monad of type (P5) whose cohomology $\E$ is a stable bundle. The first column of the matrix representing the morphism
$$ \alpha : \op3(2)\oplus2\cdot\op3\oplus2\cdot\op3(-1)\oplus\op3(-3) \to 2\cdot\op3(2) $$
consists of zeros, by minimality. Let $\jmath$ denote the inclusion of $\op3(2)$ into the first summand of $\op3(2)\oplus2\cdot\op3\oplus2\cdot\op3(-1)\oplus\op3(-3)$, so that $\alpha\circ \jmath=0$. This means that $j$ is a syzygy of degree 2, contradicting Lemma \ref{no syz deg 0}.
\end{proof}

Our next step is to establish the existence of the remaining types of minimal monads for $c_2=6$.

\begin{Theorem} \label{thm:p3+p4}
There are minimal monads of type (P3) and (P4), as labeled in Table \ref{table:c2=6}, whose cohomology bundle is stable.
\end{Theorem}
\begin{proof}
Our argument is pretty straightforward: we provide explicit examples of monads of the desired type, which were found with the help of Macaulay2. 

A minimal monad of type (P3)
\begin{equation*}
\op3(-2)\oplus\op3(-3)\stackrel{\beta}{\rightarrow}\op3(1)\oplus2\cdot\op3\oplus2\cdot\op3(-1)\oplus\op3(-2)
\stackrel{\alpha}{\rightarrow}\op3(2)\oplus\op3(1),
\end{equation*} 
is given by the morphisms
$$\alpha=\left(
\begin{array}{cccccc}
x& z^2&w^2& y^3& 0& x^4\\
0& y& x& z^2&w^2&y^3\\ 
\end{array} \right)
~\mbox{ and }~\beta=\left(
\begin{array}{cc}
-y^3&-x^4\\
w^2&0\\
-z^2&-y^3\\
x&w^2\\
-y&-z^2\\
0&x\\
\end{array}
\right).$$
One can check that $\alpha$ is surjective, $\beta$ is injective and $\alpha\circ\beta=0$. In addition, one can also check (with the help of Macaulay2) that $\alpha$ does not admit a syzygy of degree 0, thus Lemma \ref{no syz deg 0} implies that the cohomology of the monad is indeed stable.

For the possibility (P4), namely
\begin{equation*}
2\cdot\op3(-3) \stackrel{\beta}{\rightarrow} 
3\cdot\op3(-2)\oplus3\cdot\op3(1)
\stackrel{\alpha}{\rightarrow} 2\cdot\op3(2),
\end{equation*} 
we consider the morphisms
$$\alpha=\left(
\begin{array}{cccccc}
y& x&z& 0& w^4& x^4\\
w& z& y& x^4&0&w^4\\ 
\end{array} \right)
~\mbox{ and }~\beta=\left(
\begin{array}{cc}
0&x^4\\
w^4&0\\
x^4&w^4\\
-y&-w\\
-x&-z\\
-z&-y\\
\end{array}
\right),$$
Again, one can check that $\alpha$ is surjective, $\beta$ is injective and $\alpha\circ\beta=0$. In addition, one can also check (with the help of Macaulay2) that $\alpha$ does not admit a syzygy of degree 0, thus Lemma \ref{no syz deg 0} implies that the cohomology of the monad is indeed stable.
\end{proof}

\begin{prop}\label{propa1}
There are no monads of type (P15). 
\end{prop}
\begin{proof}
We suppose that the monad (P15) exists; as in the proof of the Theorem \ref{negative} the matrix of $\alpha$ has the form 
$$\alpha=\left(\begin{array}{cccccc}
0&0&0&q_1&q_2&q_3\\
\ast&\ast&\ast&\ast&\ast&\ast\\
\end{array}\right),$$
where $q_i\in\Gamma\left(\O(3)\right)$, $i=1,2,3$. If we take a point $p\in \mathcal{V}(q_1)\cap \mathcal{V}(q_2)\cap \mathcal{V}(q_3)$, then $\rank \alpha(p)<2$ which contradicts the fact that $\alpha$ is surjective. Therefore, monads of type (P15) do not exist.
\end{proof}

Next, we establish the existence of the remaining positive minimal monads listed in Table \ref{table:c2=8}, besides the Ein and Hartshorne type monad. 

\begin{Theorem}\label{monads c2=8}
There are minimal monads of type (P8), (P9), (P10), (P12), (P14) and (P17), as labeled in Table \ref{table:c2=8}, whose cohomology bundle is stable.
\end{Theorem}
\begin{proof}
First, we use Lemma \ref{lema48} to prove the existence of the positive minimal monads of type (P9) and (P12). 

Indeed, Theorem \ref{thm:p3+p4} guarantees the existence of a stable bundle $\EE$ with $c_1(\EE)=-1$, $c_2(\EE)=6$ and $H^0(\EE(2))\neq0$, given as cohomology of a positive minimal monad of type (P3). So let $s\in H^0(\EE(2))$ such that $X:=(s)_0$ is a curve, and let $C$ be a complete intersection curve of type $(2, 1)$ disjoint of $X$. Lemma \ref{lema48} then allows us to conclude that there is a stable rank 2 bundle $\EE'$ 
satisfying $c_1(\EE')=-1, c_2(\EE')=8$ that is given as the cohomology of a monad of type (P9).

In order to obtain a minimal monad of type (P12), we can just repeat the arguments above replacing (P3) by (P4).

The other cases listed in the statement of the present theorem are obtained via the straightforward method already used in Theorem \ref{thm:p3+p4}, again with the help of Macaulay2. 

An explicit monad of type (P8) is given by the morphisms  
$$\alpha=\left(
\begin{array}{cccccc}
x^2& y^2&z^2& 0& w^3& 0\\
w& z& w& x^2&0&y^2\\ 
\end{array} \right)
~\mbox{ and }~\beta=\left(
\begin{array}{cc}
y^2&-w^3\\
-x^2&0\\
0&w^3\\
z&y^2\\
0&x^2-z^2\\
-w&-x^2\\
\end{array}
\right),$$
and we can see that $\alpha$ is surjective,  $\beta$ is injective and $\alpha\circ\beta=0$. 
A monad of type (P10) is obtained by considering
$$\alpha=\left(
\begin{array}{cccccc}
z& x&y^2& z^3& w^4& 0\\
y& z& x^2& 0&y^4&w^4\\ 
\end{array} \right)
~\mbox{ and }~\beta=\left(
\begin{array}{cc}
xy^3+y^2z^2&y^3z+x^2z^2+w^4\\
-y^3z+w^4&0\\
0&-yz^2\\
-y^2&-x^2\\
-x&-z\\
-z&-y\\
\end{array}
\right).$$

Next, a monad of type (P14) is given by the morphisms
$$\alpha=\left(
\begin{array}{cccccc}
x& y^3&x^3& 0& w^4& z^6\\
0& x& y& z^2&0&w^4\\ 
\end{array} \right)
~\mbox{ and }~\beta=\left(
\begin{array}{cc}
-x^2z^2-w^4&-z^6\\
0&-w^4\\
z^2&0\\
-y&0\\
x&y^3\\
0&x\\
\end{array}
\right).$$

Finally, a monad of type (P17) is obtained by considering the morphisms
$$\alpha=\left(
\begin{array}{cccccc}
x& y^2&x^2& 0& w^5& z^6\\
0& x& y& z^4&0&w^5\\ 
\end{array} \right)
~\mbox{ and }~\beta=\left(
\begin{array}{cc}
xz^4-w^5&-z^6\\
0&-w^5\\
-z^4&0\\
y&0\\
x&y^2\\
0&x\\
\end{array}
\right).$$
\end{proof}

Hartshorne and Rao posed the following question in \cite[(Q2), p. 806]{HR91}.
\begin{center}
	\textit{Are the $3$ conditions \ref{itema1}--\ref{itema4} sufficient for a stable bundle to exist with the spectrum $\X$?}
\end{center}
They gave a positive answer to this question when $c_1=0$ and $c_2\leq19$, leaving it open for large values of $c_2$.

We will now settle the question, arguing that it has a negative answer in general. More precisely:

\begin{prop}\label{prop4}
The multiset $\mathcal{X}_6^8=\{r_0r_1^3\}$ satisfies the properties \ref{itema1}--\ref{itema4}, but there is no stable rank $2$ bundle 
on $\mathbb{P}^3$ with $c_1=-1$ having $\X_6^8$ as spectrum. 
\end{prop}

\begin{proof}
If there is a stable rank 2 bundle $\EE$ whose spectrum is $\mathcal{X}_6^8$, then according to Tables \ref{table:c2=8} and \ref{table:5}, $\EE$ is cohomology of one of the following minimal monads: (P18), (N12) or (N13). However, the non-existence of these monads was proven in the Lemma \ref{lema1}, Remark \ref{obs1} and Proposition \ref{propa3}, respectively. 
\end{proof}


\section{Moduli Space of stable rank $2$ vector bundle on $\mathbb{P}^3$} \label{se:moduli}

The goals of this section are to compute the dimension of the families of stable rank $2$ vector bundles with odd determinant given as cohomology of homotopy free minimal monads, and then use this information to characterize new components of the moduli spaces $\mathcal{B}(-1,6)$ and $\mathcal{B}(-1,8)$, different from the ones already described in the literature.


Define $\mathcal{P}(\boldsymbol{a};\boldsymbol{b})$ to be the family of minimal monads of as in display \eqref{eq6}, where $\boldsymbol{a}=(a_1\dots a_s):=(a_1^{r_1},\cdots, a_k^{r_k})$ and $\boldsymbol{b}=(b_1,\dots,b_{s+1}):=(b_1^{t_1},\cdots, b_w^{t_w})$; assume that $b_{s+1}<a_1$, so that every monad in $\mathcal{P}(\boldsymbol{a};\boldsymbol{b})$ is homotopy free.

Let $\mathcal{V}(\boldsymbol{a};\boldsymbol{b})$ be the set of isomorphism classes of stable rank $2$ bundles on $\mathbb{P}^3$ with odd determinant which are given as cohomology of a monad in $\mathcal{P}(\boldsymbol{a};\boldsymbol{b})$. We want to show that $\mathcal{P}(\boldsymbol{a};\boldsymbol{b})$ has the structure of a quasi-projective variety, and that $\mathcal{V}(\boldsymbol{a};\boldsymbol{b})$ can be regarded as a quotient of $\mathcal{P}(\boldsymbol{a};\boldsymbol{b})$ by the action of a suitable algebraic group.

Set, as above, $\mathcal{B}=\bigoplus_{j=1}^{s+1}\Big(\op3(b_j)\oplus\op3(-b_j-1)\Big)$ and $\mathcal{A}=\bigoplus_{i=1}^{s}\op3(-a_i-1)$. Let
$$ H := \{ \beta:\mathcal{A}^\vee(-1)\longrightarrow\mathcal{B} ~|~ \beta \mbox{ is locally left invertible }\}, $$
and note that, from the diagram in display \eqref{diag-q} and the equation in display \eqref{eq-beta},
$$ \mathcal{P}(\boldsymbol{a};\boldsymbol{b}) =\{ \beta \in H ~|~ \beta^*(-1)\circ\Omega\circ\beta=0 \}. $$
This means that $\mathcal{P}(\boldsymbol{a};\boldsymbol{b})$ is represented as a locally closed subset of $\Hom(\mathcal{A}^\vee(-1),\mathcal{B})$, which is an affine space.
Since $\beta^*(-1)\circ \Omega\circ\beta$ is skew symmetric as a morphism from $\mathcal{A}^\vee(-1)$ to $\mathcal{A}$, we conclude that 
$$ \dim\mathcal{P}(\boldsymbol{a};\boldsymbol{b}) = \dim H - \dim W $$
where $W$ denotes the subspace of skew symmetric bilinear forms on $\mathcal{A}(-1)$.

Next, we consider the following two groups:
$$ G := \{\varphi\in\End(\mathcal{B}) ~|~ \varphi^*(-1)\circ\Omega\circ\varphi=\Omega\},~ \mbox{ and } $$
$$ \gl(\mathcal{A}) := \{ u:\mathcal{A}\longrightarrow\mathcal{A} ~|~ u \mbox{ is an isomorphism }\}. $$
They act on $\mathcal{P}(\boldsymbol{a};\boldsymbol{b})$ as follows
\begin{equation}\label{act}
(u,\varphi)\cdot\beta = \varphi^{-1}\circ\beta\circ u^*(-1) ~, \mbox{ for } u\in\gl(\mathcal{A}), \mbox{ and } \varphi\in G.
\end{equation}
Clearly, the subgroup $\pm(\id,\id)\subset\big( \gl(\mathcal{A})\times G \big)$ acts trivially on $\mathcal{P}(\boldsymbol{a};\boldsymbol{b})$, so we set
$$ G_0 := \big( \gl(\mathcal{A})\times G \big) \Big/ \pm(\id,\id). $$
Note that $G_0$ acts freely on $\mathcal{P}(\boldsymbol{a};\boldsymbol{b})$. Indeed, if $\varphi^{-1}\circ\beta\circ u^*(-1)=\beta$, then the pair $(u,\varphi)$ induces the following automorphism of the corresponding monad
\begin{equation} \label{isom-monad} \begin{aligned}
\xymatrix{
\mathcal{A}^\vee(-1) \ar[r]^{\beta} \ar[d]_{u^*(-1)^{-1}} & \mathcal{B} \ar[r]^{\alpha} \ar[d]^{\varphi^{-1}} & \mathcal{A} \ar[d]^{u} \\
\mathcal{A}^\vee(-1) \ar[r]^{\beta} & \mathcal{B} \ar[r]^{\alpha} & \mathcal{A}
}  \end{aligned}    
\end{equation}
where $\alpha=\beta^*(-1)\circ\Omega$. Since the monad is homotopy free and its cohomology bundle is stable, every automorphism must be a multiple of the identity, thus $(u,\varphi)=\lambda\cdot(\id,\id)$; however, $\varphi$ is orthogonal, thus $\lambda^2=1$, as desired.

It is not difficult to see that two monads in $\mathcal{P}(\boldsymbol{a};\boldsymbol{b})$ are isomorphic if and only if they belong to the same orbit of this action. Consequently, the cohomology bundles form monads in the same orbit are isomorphic, providing a map
$$ \mathcal{P}(\boldsymbol{a};\boldsymbol{b}) \big/ 
G_0 \longrightarrow \mathcal{V}(\boldsymbol{a};\boldsymbol{b}). $$
We now argue that this is an bijective. Indeed, surjectivity comes from the very definition of the set $\mathcal{V}(\boldsymbol{a};\boldsymbol{b})$, while injectivity follows from the fact that monads in different orbits have non-isomorphic cohomology bundles. It follows that $\mathcal{V}(\boldsymbol{a};\boldsymbol{b})$ can be regarded as a quasi-projective variety parametrizing a family of (isomorphism classes of) stable rank 2 bundles. Therefore, we obtain an injective modular morphism 
$$ \Psi ~:~ \mathcal{V}(\boldsymbol{a};\boldsymbol{b}) \to \mathcal{B}(-1,c_2), $$
where $c_2$ is given in terms of $(\boldsymbol{a};\boldsymbol{b})$ according to the formula in display \eqref{eq1}.

The dimension of $\operatorname{im}(\Psi)$ coincides with $\dim\mathcal{V}(\boldsymbol{a};\boldsymbol{b})$, and this is computed by the following formula
\begin{equation} \label{dim formula}
\dim \mathcal{V}(\boldsymbol{a};\boldsymbol{b}) = \dim H - \dim W - \dim\gl(\mathcal{A}) - \dim G ;
\end{equation}
the first two terms give $\dim\mathcal{P}(\boldsymbol{a};\boldsymbol{b})$, and then we discount the dimensions of the groups acting on $\mathcal{P}(\boldsymbol{a};\boldsymbol{b})$. 

To compute the dimension of $G$, we observe that its Lie algebra $\mathfrak{G}$ is the set of invertible elements $\sigma\in\End(\mathcal{B})$ such that
$\sigma^*(-1)q-q\sigma=0$ which implies $q\sigma$ symmetric.
Thus we can write $\mathfrak{G}=q^{-1}S,$ where 
$S\subset\Hom(\mathcal{B}, \mathcal{B}^\vee(-1))$ is the subspace of symmetric bilinear forms on $\mathcal{B}$. Therefore  
$$\dim G=\dim\mathfrak{G}=\dim S.$$

With the help of the dimension formula in display \eqref{dim formula}, we list in Table \ref{table:1} the dimensions of the families $\mathcal{V}(\boldsymbol{a};\boldsymbol{b})$ of stable rank 2 bundles with odd determinant and $c_2\leq8$ on $\mathbb{P}^3$ that are given as cohomology of a minimal, free homotopy monad.

\begin{table}[h!]
\begin{center}
\begin{tabular}{|c|c|c|c|c|c|c|c|} 
\hline
Spectrum & $\boldsymbol{b}$& $\boldsymbol{a}$&$w$& $g$& $s$&$h$&$\dim \mathcal{V}(\boldsymbol{a};\boldsymbol{b})$\\ 
\hline \hline
$\X^6_1=\{r_0^3\}$ & $0^4$&$1^3$ &$60$&$9$&$56$&$168$& $\textbf{43}$\\
\hline
$\X^6_2=\{r_0^2r_1\}$ &$0^2$&$2$&$0$&$1$&$16$&$60$&$\textbf{43}$\\
\hline
$\X^6_3=\{r_0r_1^2\}$ &$1^3$&$2^2$&$56$&$4$&$129$&$234$& $45$\\
\hline
$\X^6_4=\{r_0r_1r_2\}$ & $0, 2$&$3$ &$0$&$1$&$92$&$143$&$50$\\
\hline\hline
$\X^8_1=\{r_0^4\}$ & $0^5$&$1^4$&$120$&$16$&$85$&$280$&$\textbf{59}$\\
\hline
$\X^8_2=\{r_0^3r_1\}$ &$0^3$&$1, 2$&$35$&$6$&$33$&$132$&$58$\\
\hline
$\X^8_3=\{r_0^2r_1^2\}$ &$0, 1^2$&$2^2$&$56$&$4$&$97$&$216$&$\textbf{59}$\\
\hline
$\X^8_5=\{r_0r_1^2r_2\}$ & $1^2$&$3$  &$0$& $1$& $64$& $132$& $67$\\
\hline
$\X^8_7=\{r_0r_1r_2r_3\}$ &$0, 3$& $4$ &$0$&$1$&$181$&$260$&$78$\\
			\hline
		\end{tabular}
		\medskip
		\caption{Computation of the dimensions of the families of stable bundles $\mathcal{V}(\boldsymbol{a};\boldsymbol{b})$; we set $h:=\dim H$, $w:=\dim W$, $g:=\dim \gl(\mathcal{A})$, and $s:=\dim G=\dim S$. The families with dimension equal to the expected one are marked in bold.}
		\label{table:1}
	\end{center}
\end{table}


With these results in mind, we are finally ready to state the main results of this section.

\begin{Theorem}\label{m(-1,6)}
The moduli scheme $\mathcal{B}(-1,6)$ has at least four irreducible components:
\begin{enumerate}
\item the Hartshorne component $M_1$, containing the family $\mathcal{V}(1^3;0^4)$ as an open set, with dimension equal to $43$;
\item two Ein components $M_2$ and $M_3$, whose general point corresponds to an element of the families $\mathcal{V}(2;0^2)$ and $\mathcal{V}(3;0, 2)$, with dimensions equal to $43$ and $50$, respectively.
\item a new component $M_4$, of dimension greater than or equal to $45$ containing the family $\mathcal{V}(2^2;1^3)$.
\end{enumerate}
\end{Theorem}

\begin{proof}
It was shown by Hartshorne \cite[Section 4]{H78} that $M_1$ is an irreducible component of $\mathcal{B}(-1,6)$ of dimension $45$ while Ein in \cite{Ein88} proved that $M_2, M_3$ are irreducible components of $\mathcal{B}(-1,6)$. The dimensions of these components are also computed in the given references.

Since the family $\mathcal{V}(2^2;1^3)$ has dimension equal to $45$, it cannot contained 
in the components $M_1$ and $M_2$; we will show that this family is not contained in $M_3$.

Indeed, assume that $\mathcal{V}(2^2;1^3)\subset M_3$. From definition of spectrum, we get $h^1(\mathcal{E}(-3))=1$ for $\mathcal{E}\in M_3$; then the inferior semi-continuity of the dimension of the cohomology groups of coherent sheaves, we get that
\begin{equation}\label{eq7}
h^1(\mathcal{F}(-3))\geq1,
\end{equation}
for every $\mathcal{F}\in M_3$. However, every 
$\mathcal{E}\in \mathcal{V}(2^2;1^3)$ satisfies $h^1(\mathcal{E}(-3))=0$, thus it cannot be a limit of bundles in $M_3$. It follows that $\mathcal{V}(2^2;1^3)$ must be 
contained in a new irreducible component of $\mathcal{B}(-1,6)$, whose dimension is at least $45$.
\end{proof}

\begin{Remark}
The existence of other irreducible components in $\mathcal{B}(-1,6)$, beyond the ones presented in Theorem \ref{m(-1,6)} requires a careful study of the monads of type (P3), (N1) and (N2), which are not homotopy free. At the moment, we do not know how to dimensions of the corresponding families of bundles.
\end{Remark}

Regarding the moduli scheme $\mathcal{B}(-1,8)$, only three components are known: 
\begin{enumerate}
\item the Hartshorne component $M_1$, with dimension equal to $59$, whose general point corresponds to the cohomology of a monad in $\mathcal{P}(1^4;0^5)$;
\item two Ein components $M_2$ and $M_3$, whose general point corresponds to the cohomology of a monad in $\mathcal{P}(3;1^2)$ and $\mathcal{P}(4;0, 3)$, with dimensions equal to $67$ and $78$, respectively.
\end{enumerate}
Our next result establishes the existence of a fourth irreducible component.

\begin{Theorem}\label{m(-1,8)}
In addition to the three components $M_1$, $M_2$, and $M_3$ mentioned above, the moduli scheme $\mathcal{B}(-1,8)$ possesses at least one more irreducible component of dimension larger than or equal to $59$ which contains the family $\mathcal{V}(2^2; 0, 1^2)$.
\end{Theorem}
\begin{proof}
A generic point $\mathcal{F}_2\in M_2$ satisfies $h^1(\mathcal{F}_2(-4))=1$, while 
in a generic point $\mathcal{F}_3\in M_3$ satisfies $h^1(\mathcal{F}_3(-3))=1$. Since every $\EE\in \mathcal{V}(2^2; 0, 1^2)$ has $h^1(\EE(k))=0$ for $k\le-3$, it follows from semi-continuity that $\mathcal{V}(2^2; 0, 1^2)$ cannot be contained neither in $M_2$ nor in $M_3$. 

Since $\dim \mathcal{V}(2^2; 0, 1^2)=\dim M_1=59$ and the generic point in $M_1$ corresponds to a Hartshorne bundle, we also conclude that the family $\mathcal{V}(2^2; 0, 1^2)$ cannot be contained in $M_1$. Therefore, $\mathcal{V}(2^2; 0, 1^2)$ must lie in a new irreducible component of $\mathcal{B}(-1,8)$.
\end{proof}

\begin{Remark}\label{p8}
The family $\mathcal{V}(1, 2; 0^3)$ has dimension equal to $58$, and therefore it does not define another irreducible component of $\mathcal{B}(-1,8)$. Since $h^1(\EE(-2))=1$ for every $\EE\in \mathcal{V}(1, 2; 0^3)$, this family must be contained in an irreducible component whose generic point corresponds to a bundle $\mathcal{F}$ satisfying $h^1(\mathcal{F}(-2))\le1$. In the table below, we summarize of possible values of $h^1(\mathcal{F}(-2))$ when $\mathcal{F}$ is a stable rank 2 bundle with $c_1(\mathcal{F})=-1$ and $c_1(\mathcal{F})=8$, according to it possible spectra:
\begin{table}[h!]
\centering
\begin{tabular}{|c|c|c|c|c|c|c|}
\hline
 $\mathcal{X}(\mathcal{F})$&$\{r_0^4\}$& $\{r_0^3r_1\}$& $\{r_0^2r_1^2\}$& $\{r_0^2r_1r_2\}$& $\{r_0r_1^2r_2\}$&$\{r_0r_1r_2r_3\}$\\
 \hline
 $h^1(\mathcal{F}(-2))$& $0$& $1$& $2$& $3$& $4$& $6$\\
 \hline
   \end{tabular}
  \caption{Dimension of $H^1(\mathcal{F}(-2))$ for $\mathcal{F}\in\mathcal{B}(-1,8)$.}
 \label{table:6}
\end{table}

Therefore, there are two possibilities: either all the sheaves $\EE$ with spectrum equal to $\mathcal{X}^8_2=\{r_0^3r_1\}$, including those in the family $\mathcal{V}(1, 2; 0^3)$, define a new irreducible component of $\mathcal{B}(-1,8)$, or all such sheaves lie in the Hartshorne component. To settle this problem, one must analyze the monads of type (P9), (N4) and (N5), whose cohomology bundles also have spectrum equal to $\mathcal{X}^8_2=\{r_0^3r_1\}$.
\end{Remark}

\begin{Remark}\label{p13}
As shown in the Table \ref{table:c2=8}, the moduli scheme $\mathcal{B}(-1,8)$ admits an irreducible component (namely, the Ein component $M_2$) which contains two families
of bundles with different spectra; these bundles are cohomology of the monads of type (P13) and (P16). In this case, a general point of $M_2$ corresponds to a bundle given as the cohomology of a monad of type (P16). 
\end{Remark}


\bibliographystyle{amsalpha}

\end{document}